\documentclass[preprint,10pt]{elsarticle}
\usepackage{amssymb}
\usepackage{subcaption}
\usepackage{atbegshi}
\AtBeginDocument{\AtBeginShipoutNext{\AtBeginShipoutDiscard}}
\usepackage[english,french]{babel}
\usepackage[utf8]{inputenc} %les accents
\usepackage{times}
\usepackage{amsthm}
\usepackage{lmodern}
\usepackage[a4paper]{geometry}
\usepackage{epstopdf}
\usepackage{mathrsfs}
\usepackage{latexsym,amssymb,amsfonts}
\usepackage{enumerate}
\usepackage{graphicx}
 \usepackage{url}

\usepackage{graphicx}
\usepackage{alltt}
\usepackage{amsmath}
\usepackage[labelfont=bf]{caption}
\usepackage{hyperref}
\hypersetup{
    colorlinks,
    citecolor=red,
    filecolor=red,
    linkcolor=red,
    urlcolor=red
}

\newtheorem{thm}{Theorem}[section]

\newtheorem{lem}[thm]{Lemma}

\theoremstyle{definition}

\newtheorem{rem}[thm]{Remark}
\newtheorem{ex}[thm]{Example}

\def\R{\mathbb{R}}

\def\P{\mathbb P}

\def\E{\mathbb E}

\journal{.}
\usepackage{geometry}
\begin{document}
\selectlanguage{english}
\begin{frontmatter}
%\begin{center}

%\vspace*{3cm}
%Nonlinear Analysis: Modelling and Control.\\
%\copyright\ Vilnius University\\[24pt]
%\vspace*{1cm}
%\LARGE
\title{Dynamic characterization of a stochastic SIR  infectious disease model with dual perturbation}
%Modelling, Systems and Technologies of Information
%\\[6pt]
%\vspace*{1cm}
%\small
\author{Driss Kiouach\footnote{Corresponding author.\\
E-mail addresses: \href{d.kiouach@uiz.ac.ma}{d.kiouach@uiz.ac.ma} (D. Kiouach), \href{yassine.sabbar@edu.uiz.ac.ma}{yassine.sabbar@edu.uiz.ac.ma} (Y. Sabbar)} and  Yassine Sabbar }
%%\\[6pt]
\address{LPAIS Laboratory, Faculty of Sciences Dhar El Mahraz, Sidi Mohamed Ben Abdellah University, Fez, Morocco.}  
%\\[6pt]
%Received: -- January 2018.
%\end{center}
\vspace*{1cm}

\begin{abstract}
This paper presents an SIR epidemic model with two different types of perturbations: white and Lévy noises. We consecrate to develop a mathematical method to obtain the asymptotic properties of the perturbed model. We use the comparison theorem, mutually exclusive possibilities lemma, and some new techniques of the stochastic differential systems to discuss the following characteristics: persistence in the mean, ergodicity, and extinction of the disease. Finally, numerical simulations about different perturbations are realized to confirm the obtained theoretical results. \vskip 2mm

\textbf{Keywords:} Asymptotic behaviour;  SIR epidemic model; white noise; Lévy noise; stationary distribution; persistence, extinction.

\textbf{Mathematics Subject Classification:} 92B05; 93E03; 93E15.
\end{abstract}
\end{frontmatter}
\
\nocite{2009ProcDETAp}
%\cite{42,1})
%\cite{3,39,40,41}
\section{Introduction}
The mathematical models are largely used in order to describe and control the dissemination of diseases into a population \cite{1}. It will continue to be one of the vigorous themes in mathematical biology due to its significance \cite{2}. The SIR epidemic model with mass action rate is a standard model among many mathematical models that present the first tentative to understand the transmission mechanisms of infectious epidemics \cite{3}. The traditional deterministic SIR epidemic model is described by the following ordinary differential equation: \cite{4}:
\begin{flalign}
&\begin{cases}
\overset{.}{S}(t)=A-\mu_1 S(t)-\beta S(t)I(t),\\
\overset{.}{I}(t)=\beta S(t)I(t) -(\mu_2+\gamma)I(t),\\
\overset{.}{R}(t)=\gamma I(t)-\mu_1 R(t),
\end{cases}&
\label{s1}
\end{flalign}
with initial data $S(0)=S_0> 0$, $I(0)=I_0>0$, $R(0)=R_0> 0$. $S(t)$ denotes the number of individuals sensitive to the disease, $I(t)$ denotes the number of contagious individuals and $R(t)$ denotes the number of recovered individuals with full immunity. The positive parameters of the deterministic model (\ref{s1}) are given in the table \ref{t1}. The basic reproduction number $\mathcal{R}_0=\frac{\beta A}{\mu_1(\mu_2+\gamma)}$ is the threshold of the system (\ref{s1}) for a disease to persist or extinct \cite{5}. If $\mathcal{R}_0\leq 1$, then the system (\ref{s1}) has only the disease-free equilibrium $P^0$ which is globally asymptotically stable; this means that the disease will extinct. If $\mathcal{R}_0 > 1$, $P^0$ will become unstable, therefore there exists a globally asymptotically stable equilibrium $P^*$; this means that the disease will persist.
\begin{center}
\begin{tabular}{ll}
\hline
Parameters \hspace*{0.5cm} & Interpretation \\
\hline
$A$ & The recruitment rate corresponding to births and immigration.  \\
$\mu_1$ & The natural mortality  rate. \\
$\beta$ & The transmission rate from infected to susceptible individuals. \\
$\gamma$ & The rate of recovering. \\
$\mu_2=\mu_1+\alpha$ & The general mortality rate, where $\alpha>0$ is the disease-related death rate. \\
\hline
\end{tabular}
\captionof{table}{Biological meanings of the parameters in model (\ref{s1}).}
\label{t1}
\end{center}
Taking the stochastic disturbances into account, many research papers have analyzed the following perturbed model:
\begin{flalign}
&\begin{cases}
dS(t)=\big(A-\mu_1 S(t)-\beta S(t)I(t)\big)dt-\sigma S(t)I(t)dW(t),\\
dI(t)=\big(\beta S(t)I(t) -(\mu_2+\gamma)I(t)\big)dt+\sigma S(t)I(t)dW(t),\\
dR(t)=\big(\gamma I(t)-\mu_1 R(t)\big)dt,
\end{cases}&
\label{s12}
\end{flalign}
where $W(t)$ is the standard Brownian motion defined on a complete probability space $(\Omega,\mathcal{F},\{\mathcal{F}_t\}_{t\geq 0},\P)$ with a filtration $\{\mathcal{F}_t\}_{t\geq 0}$ satisfying the usual conditions containing all the random variables that we meet in this paper. $\sigma$ is the intensity of environmental white noise. In the following, we present some results on the dynamics of the model (\ref{s12}):
\begin{enumerate}
\item In \cite{51}, the authors showed that the disease-free equilibrium of the model (\ref{s12}) is locally asymptotically stochastically stable under a suitable condition. From the numerical simulation, they founded the value of the stochastic threshold.
\item In \cite{52}, the authors proved many classes of stochastic stability by using the Lyapunov approach. They also studied the asymptotic character of the model around the endemic point of the deterministic model (\ref{s1}).
\item In \cite{53}, the authors investigate the threshold behaviour of the model (\ref{s12}) which determines the extinction or the persistence of the epidemic.
\end{enumerate}
Besides the above-mentioned perturbation, the deterministic model (\ref{s1}) can be perturbed by assuming that the white noise is directly proportional to $S(t)$, $I(t)$ and $R(t)$.  By this method, the model (\ref{s1}) will be rewritten as the following form:
\begin{flalign}
&\begin{cases}
dS(t)=\big(A-\mu_1 S(t)-\beta S(t)I(t)\big)dt+\sigma_1 S(t)dW_1(t),\\
dI(t)=\big(\beta S(t)I(t) -(\mu_2+\gamma)I(t)\big)dt+\sigma_2 I(t)dW_2(t),\\
dR(t)=\big(\gamma I(t)-\mu_1 R(t)\big)dt+\sigma_3 R(t)dW_3(t),
\end{cases}&
\label{s13}
\end{flalign}
where  $W_i(t)$ $(i=1,2,3)$ independent standard Brownian motions and $\sigma_i$ $(i=1,2,3)$ are the intensities of environmental white noises. There are numerous significant works that analyzed the dynamics of the  stochastic model (\ref{s13}). For instance:
\begin{enumerate}
\item In \cite{54}, the authors investigate the asymptotic behavior of the model (\ref{s13}) around the disease-free equilibrium of the deterministic model (\ref{s1}).
\item In \cite{55}, the authors analyze the long-time behavior of the stochastic SIR epidemic model (\ref{s13}). Precisely, they discussed the convergence of densities of the solution in $L^1$.
\end{enumerate}
On the other hand, epidemic models may face environmental perturbations, such as earthquakes, hurricanes, floods, etc. These phenomena cannot be modeled by the stochastic systems (\ref{s12}) and (\ref{s13}). To explain these phenomena, using a compensated Poisson process into the population dynamics provides an appropriate and more realistic model. So, it is interesting to treat differential systems with Lévy noise.
According to the Lévy-Itô decomposition, the Lévy process can be decomposed into the sum of Brownian motion and a superposition of independent (centered) Poisson process with a jump size. Therefore, the deterministic model (\ref{s1}) becomes the following stochastic model:
\begin{flalign}
&\begin{cases}
dS(t)=\big(A-\mu_1 S(t)-\beta S(t)I(t)\big)dt+\sigma_1 S(t) dW_1(t)+\int_Z \eta_1(u)S(t^{-})\widetilde{N}(dt,du),\\
dI(t)=\big(\beta S(t)I(t) -(\mu_2+\gamma)I(t)\big)dt+\sigma_2 I(t) dW_2(t)+\int_Z \eta_2(u)I(t^{-})\widetilde{N}(dt,du),\\
dR(t)=\big(\gamma I(t)-\mu_1 R(t)\big)dt+\sigma_3 R(t) dW_3(t)+\int_Z \eta_3(u)R(t^{-})\widetilde{N}(dt,du),
\end{cases}&
\label{s14}
\end{flalign}
where $S(t^{-})$, $I(t^{-})$ and $R(t^{-})$ are the left limits of $S(t)$, $I(t)$ and $R(t)$, respectively. $W_{i}(t)$ $(i = 1, 2,3)$ are independent Brownian motions and $\sigma_{i}>0$ $(i = 1, 2,3)$ are their intensities. $N$ is a Poisson counting measure with compensating  martingale $\widetilde{N}$ and characteristic measure $\nu$ on a measurable subset $Z$ of $(0,\infty)$ satisfying $\nu(Z)<\infty$. $W_{i}(t)$ $(i = 1, 2,3)$ are independent of $N$. It assumed that $\nu$ is a Lévy measure such that $\widetilde{N}(dt,du)=N(dt,du)-\nu(du)dt$. The bounded function $\eta:\; Z\times\Omega\to\R$ is $\mathfrak{B}(Z)\times\mathcal{F}_t$-measurable and  continuous with respect to $\nu$. The following references are two works that have studied the dynamics of the model (\ref{s14}):
\begin{enumerate}
\item In \cite{56}, the authors showed how the jump influences the dynamic and disease-free and endemic equilibrium of the model (\ref{s14}).
\item In \cite{57}, the authors investigated the effect of the Lévy jumps on the dynamics of the model (\ref{s14}). They obtained the stochastic threshold which determines the extinction or persistence of the disease.
\end{enumerate}
This paper presents a new stochastic SIR epidemic model with two different perturbations. We merge the stochastic transmission with a discontinuous perturbed mortality rate. The stochastic variability in the epidemic transmission $\beta$ and the mortality rate $\mu_1$ are presented as a decomposition of usual white noise and the lévy process, respectively. The perturbed version corresponding to the system (\ref{s1}) can be expressed by the following form:
\begin{flalign}
&\begin{cases}
dS(t)=\big(A-\mu_1 S(t)-\beta S(t)I(t)\big)dt + \sigma_1 S(t) dW_1(t)+\int_Z \eta(u)S(t^{-})\widetilde{N}(dt,du)- \sigma_2 S(t)I(t) dW_2(t),\\
dI(t)=\big(\beta S(t)I(t) -(\mu_2+\gamma)I(t)\big)dt  + \sigma_1 I(t) dW_1(t)+\int_Z \eta(u)I(t^{-})\widetilde{N}(dt,du)+ \sigma_2 S(t)I(t) dW_2(t),\\
dR(t)=\big(\gamma I(t)-\mu_1 R(t)\big)dt+ \sigma_1 R(t) dW_1(t)+\int_Z \eta(u)R(t^{-})\widetilde{N}(dt,du).
\end{cases}&
\label{s2}
\end{flalign}

The threshold analysis and the stability of the positive equilibrium state of the stochastic epidemic models are very important. However, the system (\ref{s2}) has perturbed by not only white noise but also by Lévy noise, which makes the analysis more complex. In this paper, we devote to develop a mathematical method to analyze the dynamics of the stochastic epidemic model (\ref{s2}). We are committed to proving the persistence in the mean and the existence of an ergodic stationary distribution for the model (\ref{s2}) by using new appropriate approaches. In \cite{58}, the authors used the existence of the stationary distribution of an auxiliary stochastic differential equation for establishing the threshold expression of the stochastic chemostat model with Lévy jumps. However, the obtained threshold still unknown due to the ignorance of the explicit form of the existed stationary distribution. Without using the stationary distribution of the auxiliary process, we will exploit new techniques in order to obtain the explicit form of the threshold which can close the gap left by using the classical method. Further, we employe the Feller property and mutually exclusive possibilities lemma to derive the condition for the existence of the stationary distribution. Under the same condition, the persistence of disease occurs.  As a result, we treated a problem that is intentionally ignored in literature; it is not biologically reasonable to consider two distinct conditions for the persistence and the existence of stationary distribution of the stochastic model.  \\

The organization of this paper is as follows: In section 2, we give some related preliminaries before our analysis.  In section 3, we focus on the analysis of the stochastic characteristics of the model (\ref{s2}). Almost sufficient condition for the persistence is established. Since the stationary distribution is an important statistical characteristic, the existence of a unique stationary distribution of system (\ref{s2}) is also obtained. To complete our analysis, we give sufficient conditions for the disease extinction. Finally, some conclusions and discussions are presented to end this paper.
\section{Preliminaries} \label{sec2}
In this section, we introduce some notations  and lemmas which are used to analyze our main results. To properly study our model (\ref{s2}), we have the following fundamental assumptions on the jump-diffusion coefficients:
\begin{itemize}
\item \textbf{($A_1$)} We assume that for a given $m>0$, there exists a constant $L_{m}>0$ such that
\begin{align*}
\int_Z |F(x,u)-F(y,u)|^2\nu(du)<L_{m}|x-y|^2,\hspace*{0.1cm}\forall \;|x|\vee|y|\leq m,
\end{align*}
where $F(i,u)=i\eta(u)$.
\item \textbf{($A_2$)} For all $u\in Z$, we assume that $1+\eta(u)>0$ and $\int_Z\big( \eta(u)-\ln (1+\eta(u))\big)\nu(du)<\infty.$
\item \textbf{($A_3$)} We suppose that exists a constant $\kappa_1>0$, such that $\int_Z \big(\ln(1+\eta(u))\big)^2\nu(du)\leq\kappa_1<\infty.$
\item \textbf{($A_4$)} We suppose that exists a constant $\kappa_2>0$, such that $\int_Z \big((1+\eta(u))^{2}-\eta(u)\big)^2\nu(du)\leq\kappa_2<\infty.$
\item \textbf{($A_5$)} Assume that for some $p\geq\frac{1}{2}$, $\chi_2=\mu_1-\frac{(2p-1)}{2}\sigma_1^2-\frac{1}{2p}\ell> 0$, where $$\ell=\int_Z\big((1+\eta(u))^{2p}-1-\eta(u)\big)\nu(du)<\infty.$$
\end{itemize}

In view of the epidemiological significance and the dynamical behavior, whether the stochastic model is well-posed is the first concern thing. Therefore, to analyze the stochastic model (\ref{s2}), the first problem to be solved is the existence of a unique global positive solution, that is, there is no explosion in finite time under any positive initial value $(S(0), I(0), R(0))\in \R^{3}_{+}$. It is known that there exists a unique global solution to the stochastic models for any given initial value if the coefficients verify the local Lipschitz and the linear growth conditions. Nevertheless, the coefficients of the model (\ref{s2}) do not verify the linear growth condition, which may let the solution to explode at a finite time. The following theorem assures the well-posedness of the stochastic model (\ref{s2}).
\begin{thm}
Let assumptions ($A_1$) and ($A_2$) hold. For any initial value $(S(0),I(0),R(0))\in \R^3_{+}$, there exists a unique positive solution $(S(t),I(t),R(t))$ of the system (\ref{s2}) on $t\geq 0$, and the solution will stay in $\R^{3}_{+}$ almost surely. 
\label{thmp}
\end{thm}
The proof is somehow standard and classic (see for example \cite{16}), so we omit it. In the following, we always presume that the assumptions $(A_1)$-$(A_5)$ hold.
%The proof is similar to that in Chang etal.[19] and %hence is omitted.

\begin{lem}
Let $(S(t),I(t),R(t))$ be the solution of (\ref{s2}) with initial value $(S(0),I(0),R(0))\in\R^3_+$. Then 
\begin{enumerate}
\item $\E\big(N^{2p}(t)\big)\leq (N(0))^{2p}e^{\{-p\chi_2t\}}+ \frac{2\chi_1}{\chi_2}$;
\item $\underset{t\to +\infty}{\lim \sup} \frac{1}{t}\int^t_0\E\big(N^{2p}(t)\big)ds\leq \frac{2\chi_1}{\chi_2}$\;\; a.s.
\end{enumerate}
where $\chi_1=\underset{x>0}{\sup}\{Ax^{{2p}-1}-\frac{ \chi_2}{2}x^{2p}\}$.
\label{L1}
\end{lem}
\begin{proof}
Making use of Itô's lemma, we obtain
\begin{align*}
d(N(t))^{2p}&=2p(N(t))^{2p-1}(A-\mu_1 N(t)-rI(t))dt+p(2p-1)\sigma_1^2(N(t))^{2p}dt\\&\;\;\;+\int_Z(N(t))^{2p}\big((1+\eta(u))^{2p}-1-\eta(u)\big)\nu(du)dt+2p(N(t))^{2p-1}\sigma_1N(t)dW_1(t)\\&\;\;\;+\int_Z(N(t^{-}))^{2p}\big((1+\eta(u))^{2p}-\eta(u)\big)\tilde{\mathcal{N}}(dt,du).
\end{align*}
Then
\begin{align*}
d(N(t))^{2p}&\leq 2p(N(t))^{2p-1}\big(A-\mu_1 N(t)\big)dt+p(2p-1)\sigma_1^2(N(t))^{2p}dt+2p(N(t))^{2p-1}\sigma_1N(t)dW_1(t)\\&\;\;\;+\int_Z(N(t^{-}))^{2p}\big((1+\eta(u))^{2p}-1-\eta(u)\big)\nu(du)dt+\int_Z(N(t))^{2p}\big((1+\eta(u))^{2p}-\eta(u)\big)\tilde{\mathcal{N}}(dt,du)\\
&\leq 2p\Big\{A(N(t))^{2p-1}-\Big(\mu_1-\frac{(2p-1)}{2}\sigma_1^2-\frac{1}{2p}\int_Z\big((1+\eta(u))^{2p}-1-\eta(u)\big)\nu(du)\Big)(N(t))^{2p}\Big\}dt\\&\;\;\;+2p(N(t))^{2p-1}\sigma_1N(t)dW_1(t)+\int_Z(N(t^{-}))^{2p}\big((1+\eta(u))^{2p}-\eta(u)\big)\tilde{\mathcal{N}}(dt,du).
\end{align*}
We choose neatly $p>\frac{1}{2}$ such that
\begin{align*}
\chi_2&=\mu_1-\frac{(2p-1)}{2}\sigma_1^2-\frac{1}{2p}\int_Z\big((1+\eta(u))^{2p}-1-\eta(u)\big)\nu(du)> 0.
\end{align*}
Hence 
\begin{align*}
d(N(t))^{2p}&\leq 2p\Big\{\chi_1-\frac{\chi_2}{2}(N(t))^{2p}\Big\}dt+2p(N(t))^{2p-1}\sigma_1N(t)dW_1(t)\\&\;\;\;+\int_Z(N(t^{-}))^{2p}\big((1+\eta(u))^{2p}-\eta(u)\big)\tilde{\mathcal{N}}(dt,du).
\end{align*}
On the other hand, we have
\begin{align*}
d(N(t))^{2p}e^{p\chi_2t}&=p\chi_2(N(t))^{2p}e^{p\chi_2t}+e^{p\chi_2t}d(N(t))^{2p}\\&\leq 2p\chi_1e^{p\chi_2t}+e^{p\chi_2t}2p(N(t))^{2p-1}\sigma_1N(t^{-})dW_1(t)\\&\;\;\;+\int_Z(N(t))^{2p}\big((1+\eta(u))^{2p}-\eta(u)\big)\tilde{\mathcal{N}}(dt,du).
\end{align*}
Then by taking integrations and taking the expectations, we get
\begin{align*}
(N(t))^{2p}&\leq(N(0))^{2p}e^{-p\chi_2t}+2p\chi_1\int^t_0 e^{p\chi_2(t-s)}ds\\&\leq(N(0))^{2p}e^{-p\chi_2t}+\frac{2\chi_1}{\chi_2}.
\end{align*}
Obviously, we obtain
\begin{align*}
\underset{t\to +\infty}{\lim \sup} \frac{1}{t}\int^t_0\E(N(t))^{2p}(u)]du\leq (N(0))^{2p}\underset{t\to +\infty}{\lim \sup} \frac{1}{t}\int^t_0e^{-p \chi_2u}du+\frac{2\chi_1}{\chi_2}=\frac{2\chi_1}{\chi_2}.
\end{align*}
\end{proof}
Now, we consider the following subsystem
\begin{align}
\begin{cases}dX(t)=(A-\mu_1 X(t))dt+\sigma_1X(t^{-})dW_1(t)+\int_Z\eta(u)X(t^{-})\tilde{\mathcal{N}}(dt,du) \hspace{0.5cm}\forall t>0\\
X(0)=N(0)>0.\end{cases}
\label{s4}
\end{align}
\begin{lem}\cite{166}
Let $(S(t),I(t),R(t))$ be the positive solution of the system (\ref{s2}) with any given initial condition $(S(0),I(0),R(0))\in\R^3_+$. Let also $X(t)\in\R_+$ be the solution of the equation (\ref{s4}) with any given initial value $X(0)=N(0)\in\R_+$. Then
\begin{enumerate}
\item 
\begin{align*}
\underset{t\to\infty}{\lim}\frac{X(t)}{t}=0, \hspace{0.3cm}\underset{t\to\infty}{\lim}\frac{S(t)}{t}=0,  \hspace{0.2cm} \underset{t\to\infty}{\lim}\frac{I(t)}{t}=0,  \hspace{0.2cm}\mbox{and}\hspace{0.2cm}  \underset{t\to\infty}{\lim}\frac{R(t)}{t}=0 \hspace{0.5cm}\mbox{a.s.}
\end{align*}
\item
\begin{align*}
 &\underset{t\to\infty}{\lim}\frac{\int^t_0 X(s)dW_1(s)}{t}=0, \hspace{0.3cm}\underset{t\to\infty}{\lim}\frac{\int^t_0 S(s)dW_1(s)}{t}=0, \\&\underset{t\to\infty}{\lim}\frac{\int^t_0 I(s)dW_1(s)}{t}=0,  \hspace{0.5cm}  \underset{t\to\infty}{\lim}\frac{\int^t_0 R(s)dW_1(s)}{t}=0 \hspace{0.5cm}\mbox{a.s.}
\end{align*}
\item
 \begin{align*}
&\underset{t\to\infty}{\lim}\frac{\int^t_0\int_Z \eta(u)X(s^{-})\widetilde{N}(ds,du)}{t}=0,\hspace{0.3cm}\underset{t\to\infty}{\lim}\frac{\int^t_0\int_Z \eta(u)S(s^{-})\widetilde{N}(ds,du)}{t}=0,\\&\underset{t\to\infty}{\lim}\frac{\int^t_0\int_Z \eta(u)I(s^{-})\widetilde{N}(ds,du)}{t}=0,\hspace{0.3cm}\underset{t\to\infty}{\lim}\frac{\int^t_0\int_Z \eta(u)R(s^{-})\widetilde{N}(ds,du)}{t}=0\hspace{0.1cm}\mbox{a.s.}
\end{align*}
\end{enumerate}
\label{lem1m}
\end{lem}
\begin{lem}
Let $X(t)$ be solution of the system (\ref{s4}) with an initial value $X(0)\in\R_{+}$. Then, 
\begin{align*}
\underset{t\to\infty}{\lim}\frac{1}{t}\int^t_0X(s)ds=\frac{A}{\mu_1}\hspace{0.2cm}\mbox{a.s.}
\end{align*} 
and
\begin{align*}
\underset{t\to\infty}{\lim}\frac{1}{t}\int^t_0X^2(s)ds=\frac{2A^2}{\mu_1\chi_3}\hspace{0.2cm}\mbox{a.s.}
\end{align*}
where $\chi_3=2\mu_1-\sigma_1^2-\int_Z\Big((1+\eta(u))^2-1-\eta(u)\Big)\nu(du)>0$.
\label{lemmas}
\end{lem}
\begin{proof}
Integrating from $0$ to $t$ on both sides of (\ref{s4}) yields
\begin{align*}
\frac{X(t)-X(0)}{t}=A-\frac{\mu_1}{t}\int^t_0X(s)ds+\frac{\sigma_1}{t}\int_0^t X(s)dW_1(s)+\frac{1}{t}\int^t_0\int_Z \eta(u)X(t^{-})\widetilde{N}(ds,du).
\end{align*}
Clearly, we can derive that
\begin{align*}
\frac{1}{t}\int^t_0X(s)ds=\frac{A}{\mu_1}+\frac{\sigma_1}{\mu_1 t}\int_0^t X(s)dW_1(s)+\frac{1}{\mu_1 t}\int^t_0\int_Z \eta(u)X(t^{-})\widetilde{N}(dt,du).
\end{align*}
Hence
\begin{align*}
\underset{t\to\infty}{\lim}\frac{1}{t}\int^t_0X(s)ds=\frac{A}{\mu_1}\hspace{0.2cm}\mbox{a.s.}
\end{align*}
Applying the generalized Itô’s formula to model (\ref{s4}) leads to
\begin{align*}
dX^2(t)&=\bigg(2X(t)\Big(A-\mu_1 X(t)\Big)+\sigma_1^2X^2(t)+\int_ZX^2(t)\Big((1+\eta(u))^2-1-\eta(u)\Big)\nu(du)\bigg)dt\\&\;\;\;+2\sigma_1 X^2(t)dW_1(t)+\int_Z X^2(t^{-})\Big((1+\eta(u))^2-\eta(u)\Big)\widetilde{N}(dt,du).
\end{align*}
Integrating both sides from $0$ to $t$, yields
\begin{align*}
X^2(t)-X^2(0)&=2A\int^t_0X(s)ds-\bigg(2\mu_1-\sigma_1^2-\int_Z\Big((1+\eta(u))^2-1-\eta(u)\Big)\nu(du)\bigg)\int^t_0X^2(s)ds\\&\;\;\;+2\sigma_2\int^t_0X^2(s)dW_1(s)+\int^t_0\int_ZX^2(s)\Big((1+\eta(u))^2-\eta(u)\Big)\widetilde{N}(ds,du).
\end{align*}
Let $\chi_3=2\mu_1-\sigma_1^2-\int_Z\Big((1+\eta(u))^2-1-\eta(u)\Big)\nu(du)>0$. Therefore
\begin{align*}
\frac{1}{t}\int^t_0X^2(s)ds &=\frac{2A^2}{\mu_1\chi_3}+\frac{2\sigma_1(X^2(0)-X^2(t))}{\chi_3t}+\frac{2\sigma_1}{\chi_3 t}\int^t_0X^2(s)sW_1(s)\\&\;\;\;+\frac{2\sigma_1}{\chi_3 t}\int^t_0\int_ZX^2(s)\Big((1+\eta(u))^2-\eta(u)\Big)\widetilde{N}(ds,du),
\end{align*}
By using the same method as that in \cite{166}, assumption \textbf{($A_4$)}  and the large number theorem for martingales, we can easily verify that
\begin{align*}
\underset{t\to\infty}{\lim}\frac{1}{t}\int^t_0X^2(s)ds=\frac{2A^2}{\mu_1\chi_3}\hspace{0.2cm}\mbox{a.s.}
\end{align*}
\end{proof}
\begin{rem}
Differently to the method mentioned in (Theorem 4, \cite{sta}), we have established the value of $\underset{t\to+\infty}{\lim}\frac{1}{t}\int^t_0X(s)ds$ and $\underset{t\to+\infty}{\lim}\frac{1}{t}\int^t_0X^2(s)ds$  by using a new approach without employing the ergodic theorem.
\end{rem}
Now, we present a lemma which gives mutually exclusive possibilities for the existence of an ergodic stationary distribution to the system (\ref{s2}).
\begin{lem}[\cite{10}]
Let $\phi(t)\in \R^n$ be a stochastic Feller process, then either an ergodic probability measure exists, or 
\begin{align}
\underset{t\to \infty}{\lim }\underset{\nu}{\sup }\frac{1}{t}\int^t_0\int \P(u,x,\Sigma)\nu(dx)du=0,\hspace{0.2cm}\mbox{for any compact set}\hspace{0.2cm}\Sigma\in\R^n,
\label{imp}
\end{align}
where the supremum is taken over all initial distributions $\nu$ on $R^d$ and $\P(t,x,\Sigma)$ is the probability for $\phi(t)\in \Sigma$ with $\phi(0)=x\in \R^n$.
\label{poss}
\end{lem}
\section{Main results}

%\begin{lem}\cite{291}
%Let $h(t)>0$, $k(t)\geq 0$ and $G(t)$ be functions on $[0,+\infty)$, $c\geq 0$ and $d>0$ be constants, such that $\underset{t\to\infty}{\lim}\frac{G(t)}{t}=0$ and 
%\begin{align*}
%\ln h(t)\leq ct+k(t)-d\int^t_0h(s)ds+G(t).
%\end{align*}
%If $k(t)$ is a non-decreasing function, then 
%\begin{align*}
%\underset{t\to\infty}{\lim \sup }\frac{1}{t}\bigg(-k(t)+d\int^t_0h(s)ds\bigg)\leq c.
%\end{align*}
%\label{nlem}
%\end{lem}
The aim of the following theorem is to give the condition for the persistence in the mean of the disease and the ergodicity of the stochastic model (\ref{s2}). Define the parameter:
\begin{align*}
\mathcal{R}^s_0=\Big(\mu_2+\gamma+\frac{\sigma_1^2}{2}\Big)^{-1}\left(\frac{\beta A}{\mu_1}-\frac{A^2\sigma_2^2}{\mu_1\chi_3}-\int_Z \eta(u)-\ln(1+\eta(u))\nu(du)\right).
\end{align*}
\begin{thm}
If $\mathcal{R}^s_0>1$, then for any value $(S(0), I(0), R(0))\in\R^3_+$, the disease is persistent in the mean. That is to say 
\begin{align*}
\underset{t\to\infty}{\lim\inf}\frac{1}{t}\int^t_0 I(u)du>0\hspace{0.2cm}\mbox{a.s.}
\end{align*}
Furthermore, under the same condition, the stochastic system (\ref{s2}) admits a unique stationary distribution and it has the ergodic property.
\label{thm1}
\end{thm}

\begin{proof}
On the one hand, based on the model (\ref{s2}), we get
\begin{align*}
d\Big(S(t)+I(t)\Big)&=\Big(A-\mu_1 S(t)-(\mu_2+\gamma)I(t)\Big)dt+\sigma_1\Big(S(t)+I(t)\Big)dW_1(t)\\&\;\;\;+\int_Z \eta(u)\left(S(t^{-})+I(t^{-})\right)\widetilde{N}(dt,du).
\end{align*}
Taking integral on both sides of the last equation from $0$ to $t$, we see that 
\begin{align*}
\frac{1}{t}\Big(S(t)+I(t)-S(0)-I(0) \Big)&=A-\frac{\mu_1}{t}\int^t_0S(s)ds-\frac{(\mu_2+\gamma)}{t}\int^t_0I(s)ds+\frac{\sigma_1}{t}\int^t_0(S(s)+I(s))dW_1(s)\\&\;\;\;+\frac{1}{t}\int^t_0\int_Z \eta(u)\left(S(s^{-})+I(s^{-})\right)\widetilde{N}(ds,du).
\end{align*}
Then, one can obtain that 
\begin{align}
\frac{1}{t}\int^t_0S(s)ds=\frac{A}{\mu_1}-\frac{(\mu_2+\gamma)}{\mu_1 t}\int^t_0I(s)ds+\Phi_1(t),
\label{S(t)}
\end{align}
where
\begin{align*}
\Phi_1(t)&=\frac{\sigma_1}{\mu_1t}\int^t_0(S(s)+I(s))dW_1(s)-\frac{1}{\mu_1 t}\Big(S(t)+I(t)-S(0)-I(0) \Big)\\&\;\;\;+\frac{1}{\mu_1 t}\int^t_0\int_Z \eta(u)\left(S(s^{-})+I(s^{-})\right)\widetilde{N}(ds,du).
\end{align*}
On the other hand, applying Itô’s formula to the second equation of (\ref{s2}), we get
\begin{align}
d\ln I(t)&=\bigg(\beta S(t)-\frac{\sigma_2^2}{2}S^2(t)-\Big(\mu_2+\gamma+\frac{\sigma^2_1}{2}\Big)-\int_Z\eta(u)-\ln(1+\eta(u))\nu(du)\bigg)dt\nonumber\\&\;\;\;+\sigma_1 dW_1(t)+\sigma_2S(t)dW_2(t)+\int_Z \ln(1+\eta(u))\widetilde{N}(dt,du).
\label{205}
\end{align}
Integrating (\ref{205}) from $0$ to $t$ and then dividing $t$ on both sides, we have
\begin{align*}
\frac{1}{t}(\ln I(t)-\ln I(0))&=\frac{\beta}{t}\int^t_0S(s)ds-\frac{\sigma_2^2}{2t}\int^t_0 S^2(s)ds-\Big(\mu_2+\gamma+\frac{\sigma^2_1}{2}\Big)-\int_Z\eta(u)-\ln(1+\eta(u))\nu(du)\nonumber\\&\;\;\;+\frac{\sigma_1}{t} W_1(t)+\frac{\sigma_2}{t}\int^t_0 S(s)dW_2(s)+\frac{1}{t}\int_0^t\int_Z \ln(1+\eta(u))\widetilde{N}(ds,du).
\end{align*}
From (\ref{S(t)}), we obtain
\begin{align*}
\frac{1}{t}(\ln I(t)-\ln I(0))&=\frac{\beta A}{\mu_1}-\frac{\beta(\mu_2+\gamma)}{\mu_1 t}\int^t_0I(s)ds+\beta \Phi_1(t)-\frac{\sigma_2^2}{2t}\int^t_0 S^2(s)ds\\&\;\;\;-\Big(\mu_2+\gamma+\frac{\sigma^2_1}{2}\Big)-\int_Z\eta(u)-\ln(1+\eta(u))\nu(du)\nonumber\\&\;\;\;+\frac{\sigma_1}{t} W_1(t)+\frac{\sigma_2}{t}\int^t_0 S(s)dW_2(s)+\frac{1}{t}\int_0^t\int_Z \ln(1+\eta(u))\widetilde{N}(ds,du).
\end{align*}
Following (\ref{lemmas}) and the stochastic comparison theorem, we get
\begin{align*}
\frac{1}{t}\big(\ln I(t)-\ln I(0)\big)&\geq \bigg(\frac{\beta A}{\mu_1}-\frac{\sigma_2^2}{2t}\int^t_0 X^2(s)ds-\Big(\mu_2+\gamma+\frac{\sigma_1^2}{2}\Big)-\int_Z\eta(u)-\ln(1+\eta(u))\nu(du)\bigg)\\&\;\;\;+\beta \Phi_1(t)-\frac{\beta(\mu_2+\gamma)}{\mu_1 t}\int^t_0I(s)ds+\frac{\sigma_1}{t} W_1(t)\\&\;\;\;+\frac{\sigma_2}{t}\int^t_0 S(s)dW_2(s)+\frac{1}{t}\int_Z \ln(1+\eta(u))\widetilde{N}(ds,du).
\end{align*}
Hence, we further get
\begin{align*}
\frac{\beta(\mu_2+\gamma)}{\mu_1 t}\int^t_0I(s)ds&\geq \bigg(\frac{\beta A}{\mu_1}-\frac{\sigma_2^2}{2t}\int^t_0 X^2(s)ds-\Big(\mu_2+\gamma+\frac{\sigma_1^2}{2}\Big)-\int_Z\eta(u)-\ln(1+\eta(u))\nu(du)\bigg)\\&\;\;\;+\beta \Phi_1(t)-\frac{1}{t}(\ln I(t)-\ln I_0)+\frac{\sigma_1}{t} W_1(t)\\&\;\;\;+\frac{\sigma_2}{t}\int^t_0 S(s)dW_2(s)+\frac{1}{t}\int^t_0 \int_Z \ln(1+\eta(u))\widetilde{N}(ds,du).
\end{align*}
By the large number theorem for martingales and lemma \ref{lemmas}, we conclude that
\begin{align*}
\underset{t\to\infty }{\liminf}\frac{1}{ t}\int^t_0I(u)du&\geq \frac{\mu_1 }{\beta(\mu_2+\gamma)}\Bigg(\frac{\beta A}{\mu_1}-\frac{A^2\sigma_2^2}{\mu_1\chi_3}-\Big(\mu_2+\gamma+\frac{\sigma_1^2}{2}\Big)-\int_Z\eta(u)-\ln(1+\eta(u))\nu(du)\Bigg)\\&=\frac{\mu_1 }{\beta(\mu_2+\gamma)}\Big(\mu_2+\gamma+\frac{\sigma_1^2}{2}\Big)\Big(\mathcal{R}^s_0-1\Big)>0\hspace{0.2cm}\mbox{a.s.}
\end{align*}
This shows that the system (\ref{s2}) is persistent in the mean with probability one. In the following, based on the lemme 2.5, we will discuss the existence of a unique ergodic stationary distribution of the positive solutions to the system (\ref{s2}). Similar to the proof of lemma 3.2 in \cite{11}, we briefly verify the Feller property of the SDE model (\ref{s2}). The main purpose of the next analysis is to prove that (\ref{imp}) is impossible. Same as the above, we have
%
%\\Applying Itô's formula to $\ln I(t)$, we get
%\begin{align*}
%d\ln I(t)&=\bigg(\beta S(t)-(\mu_2+\gamma)-\frac{\sigma_1^2}{2}-\frac{\sigma_2^2}{2}S^2(t)-\int_Z \eta(u)-\ln(1+\eta(u))\nu(du)\bigg)dt\\&\;\;\;+\sigma_1dW_1(t)+\sigma_2 S(t)dW_2(t)+\int_Z\ln(1+\eta(u))\tilde{\mathcal{N}}(dt,du).
%\end{align*}
%Therefore
\begin{align*}
&d\Big\{\ln I(t)-\frac{\beta}{\mu_1}\Big(X(t)-S(t)\Big)\Big\}\\&=\bigg(\beta S(t)-(\mu_2+\gamma)-\frac{\sigma_1^2}{2}-\frac{\sigma_2^2}{2}S^2(t)-\int_Z \eta(u)-\ln(1+\eta(u))\nu(du)\bigg)dt\\&\;\;\;-\frac{\beta}{\mu_1}\Big(-\mu_1(X(t)-S(t))+\beta S(t)I(t)\Big)dt+\sigma_1dW_1(t)+\sigma_2 S(t)dW_2(t)\\&\;\;\;-\frac{\beta}{\mu_1}(X(t)-S(t))dW_1(t)-\frac{\beta\sigma_1}{\mu_1}S(t)I(t)dW_2(t)\\&\;\;\;-\frac{\beta }{\mu_1}\int_Z \eta(u)(X(t)-S(t))\tilde{\mathcal{N}}(dt,du)+\int_Z\ln(1+\eta(u))\tilde{\mathcal{N}}(dt,du).&
\end{align*}
Hence
\begin{align}
&d\Big\{\ln I(t)-\frac{\beta}{\mu_1}\big(X(t)-S(t)\big)\Big\}\\&=\bigg(\beta X(t)-(\mu_2+\gamma)-\frac{\sigma_1^2}{2}-\frac{\sigma_2^2}{2}S^2(t)-\int_Z \eta(u)-\ln(1+\eta(u))\nu(du)\bigg)dt\nonumber\\&\;\;\;-\frac{\beta^2 }{\mu_1}S(t)I(t)dt+\sigma_1dW_1(t)+\sigma_2 S(t)dW_2(t)-\frac{\beta}{\mu_1}(X(t)-S(t))dW_1(t)\nonumber\\&\;\;\;-\frac{\beta\sigma_1}{\mu_1}S(t)I(t)dW_2(t)-\frac{\beta }{\mu_1}\int_Z \eta(u)(X(t)-S(t))\tilde{\mathcal{N}}(dt,du)\nonumber\\&\;\;\;+\int_Z\ln(1+\eta(u))\tilde{\mathcal{N}}(dt,du).
\label{ito}
\end{align}
Integrating from $0$ to $t$ on both sides of (\ref{ito}) yields
\begin{align*}
&\ln\frac{I(t)}{I(0)}-\frac{\beta}{\mu_1}\Big(X(t)-S(t)\Big)+\frac{\beta}{\mu_1}\Big(X(0)-S(0)\Big)\\&=\int^t_0\bigg(\beta X(s)-\frac{\sigma_2^2}{2}S^2(s)-\Big(\mu_2+\gamma+\frac{\sigma_1^2}{2}\Big)-\int_Z \eta(u)-\ln(1+\eta(u))\nu(du)\bigg)ds\\&\;\;\;-\frac{\beta^2 }{\mu_1}\int^t_0S(s)I(s)ds+\sigma_1 W_1(t)+\sigma_2\int^t_0 S(s)dW_2(s)-\frac{\beta}{\mu_1}\int^t_0(X(s)-S(s))dW_1(s)\nonumber\\&\;\;\;-\frac{\beta\sigma_1}{\mu_1}\int^t_0S(s)I(s)dW_2(s)-\frac{\beta }{\mu_1}\int^t_0\int_Z \eta(u)(X(s^{-})-S(s^{-}))\tilde{\mathcal{N}}(ds,du)\nonumber\\&\;\;\;+\int^t_0\int_Z\ln(1+\eta(u))\tilde{\mathcal{N}}(ds,du).
\end{align*}
Then, we get
\begin{align}
\int^t_0\beta S(s)I(s)ds&=\frac{\mu_1}{\beta}\int^t_0\bigg(\beta X(s)-\frac{\sigma_2^2}{2}S^2(s)-\Big(\mu_2+\gamma+\frac{\sigma_1^2}{2}\Big)-\int_Z \eta(u)-\ln(1+\eta(u))\nu(du)\bigg)ds\nonumber\\&\;\;\;+\Big(X(t)-S(t)\Big)-\Big(X(0)-S(0)\Big)-\frac{\mu_1}{\beta}\ln\frac{I(t)}{I(0)}+\frac{\sigma_1\mu_1}{\beta} W_1(t)\nonumber\\&\;\;\;+\frac{\sigma_2\mu_1}{\beta}\int^t_0 S(s)dW_2(s)-\int^t_0(X(s)-S(s))dW_1(s)-\sigma_1\int^t_0S(s)I(s)dW_2(s)\nonumber\\&\;\;\;-\int^t_0\int_Z \eta(u)(X(s^{-})-S(s^{-}))\tilde{\mathcal{N}}(ds,du)+\frac{\mu_1}{\beta}\int^t_0\int_Z\ln(1+\eta(u))\tilde{\mathcal{N}}(ds,du).
\label{11}
\end{align}
%From the system (\ref{s2}), we have $dN(t)\leq \Big(A-\mu_1 N(t)\Big)dt-\sigma_2 N(t)dW_2(t)$.
%Applying the variation of constants formula, leads to
%\begin{align*}
%N(t)\leq N_0e^{-(\mu_1+\sigma_2^2/2)t+\sigma_2 W_2(t)}+A\int^t_0 e^{-(\mu_1+\sigma_2^2/2)(t-u)+\sigma_2(W_2(t)-W_2(u))}du.
%\end{align*}
%So we obtain
%\begin{align*}
%N(t)\leq N_0e^{\sigma_2W_2(t)}+\Bigg(\frac{A}{\mu_1+\sigma_2^2/2}\Bigg)e^{2\sigma_2\underset{u\in[0,t]}{\sup}|W_2(u)|}.
%\end{align*}
%Since $\underset{t\to\infty}{\lim}\frac{1}{t}\underset{u\in[0,t]}{\sup}|W_2(u)|=0$, we get
%\begin{align}
%\underset{t\to +\infty}{\lim \sup}\frac{1}{t}\ln\frac{I(t)}{I(0)}&\leq \underset{t\to +\infty}{\lim \sup}\frac{1}{t}\ln\frac{N(t)}{I(0)}\leq 0.
%\label{12}
%\end{align}

%Now, consider the following equation:
%\begin{align*}
%d\Big(\psi(t)-S(t)\Big)=\Big[-\mu_1\Big(\psi(t)-S(t)\Big)+\beta S(t)I(t)\Big]dt-\sigma_2\Big(\psi(t)-S(t)\Big)dW_2(t)+\sigma_1 S(t) I(t)dW_1(t).
%\end{align*}
%By virtue of the variation of constants formula, one can obtain that
%\begin{align*}
%\psi(t)-S(t)&=\Big(\psi_0-S_0\Big)e^{\{-(\mu_1+\sigma_2^2/2)t+\sigma_2W_2(t)\}}+\beta\int^t_0e^{\{-(\mu_1+\sigma_2^2/2)(u-t)-\sigma_2(W_2(t)-W_2(u))\}}S(u)I(u)du\\&\;\;\;-\sigma_1\int^t_0e^{\{-(\mu_1+\sigma_2^2/2)(u-t)-\sigma_2(W_2(t)-W_2(u))\}}S(u)I(u)dW_1(t).
%\end{align*}
%Therefore,
%\begin{align}
%\underset{t\to +\infty}{\lim \inf}\frac{1}{t} \Big(\psi(t)-S(t)\Big) \geq 0.
%\label{13}
%\end{align}
Since $\underset{t\to +\infty}{\lim \sup}\frac{1}{t}\ln\frac{I(t)}{I(0)}\leq \underset{t\to +\infty}{\lim \sup}\frac{1}{t}\ln\frac{S(t)+I(t)}{I(0)}\leq 0$ a.s. and according to the large number theorem for martingales, one can derive that
\begin{align}
&\underset{t\to +\infty}{\lim \inf} \frac{1}{t}\int^t_0 \beta S(s)I(s)ds\nonumber\\&\geq \frac{\mu_1}{\beta}\underset{t\to +\infty}{\lim \inf} \frac{1}{t}\int^t_0\bigg(\beta X(s)-\frac{\sigma_2^2}{2}X^2(s)-\Big(\mu_2+\gamma+\frac{\sigma_1^2}{2}\Big)-\int_Z \eta(u)-\ln(1+\eta(u))\nu(du)\bigg)du\nonumber\\
&=\frac{\mu_1}{\beta}\underset{t\to +\infty}{\lim } \frac{1}{t}\int^t_0\bigg(\beta X(s)-\frac{\sigma_2^2}{2}X^2(s)-\Big(\mu_2+\gamma+\frac{\sigma_1^2}{2}\Big)-\int_Z \eta(u)-\ln(1+\eta(u))\nu(du)\bigg)du.
\label{from}
\end{align}
Now from lemma \ref{lemmas}, it follows that
\begin{align}
\underset{t\to +\infty}{\lim \inf} \frac{1}{t}\int^t_0 \beta S(s)I(s)ds\nonumber&\geq \frac{\mu_1}{\beta}\times\Bigg(\frac{\beta A}{\mu_1}-\frac{A^2\sigma_2^2}{\mu_1\chi_3}-\Big(\mu_2+\gamma+\frac{\sigma_1^2}{2}\Big)-\int_Z \eta(u)-\ln(1+\eta(u))\nu(du)\Bigg)\nonumber\\&=\frac{\mu_1}{\beta}\Big(\mu_2+\gamma+\frac{\sigma_1^2}{2}\Big)\Big(\mathcal{R}_0^s-1\Big)>0\hspace{0.5cm}\mbox{a.s.}
\label{ergo1}
\end{align}
To continue our analysis, we need to set the following subsets: $\Omega_1=\{(S,I,R)\in\R^3_+|\hspace{0.1cm}S\geq \epsilon,\hspace{0.1cm}\mbox{and},\hspace{0.1cm} I\geq \epsilon\}$, $\Omega_2=\{(S,I,R)\in\R^3_+|\hspace{0.1cm}S\leq \epsilon\}$, and $\Omega_3=\{(S,I,R)\in\R^3_+|\hspace{0.1cm}I\leq \epsilon\}$ where $\epsilon>0$ is a positive constant to be determined later. Therefore, by (\ref{ergo1}), we get
\begin{align*}
\underset{t\to +\infty}{\lim \inf}\frac{1}{t}\int^t_0\E\Big( \beta S(s)I(s)\mathbf{1}_{\Omega_1}\Big)ds&\geq \underset{t\to +\infty}{\lim \inf}\frac{1}{t}\int^t_0\E\Big( \beta S(s)I(s)\Big)ds-\underset{t\to +\infty}{\lim \sup}\frac{1}{t}\int^t_0\E\Big( \beta S(s)I(s)\mathbf{1}_{\Omega_2}\Big)ds\\&\;\;\;-\underset{t\to +\infty}{\lim \sup}\frac{1}{t}\int^t_0\E\Big( \beta S(s)I(s)\mathbf{1}_{\Omega_3}\Big)ds\\&\geq \frac{\mu_1}{\beta}\Big(\mu_2+\gamma+\frac{\sigma_1^2}{2}\Big)\Big(\mathcal{R}_0^s-1\Big)-\beta \epsilon \underset{t\to +\infty}{\lim \sup}\frac{1}{t}\int^t_0\E\Big(I(s)\Big)ds\\&\;\;\;-\beta \epsilon \underset{t\to +\infty}{\lim \sup}\frac{1}{t}\int^t_0\E\Big(S(s)\Big)ds.
\end{align*}
By lemma \ref{L1}, one can see that
\begin{align*}
\underset{t\to +\infty}{\lim \inf}\frac{1}{t}\int^t_0\E\Big( \beta S(s)I(s)\mathbf{1}_{\Omega_1}\Big)ds&\geq \frac{\mu_1}{\beta}\Big(\mu_2+\gamma+\frac{\sigma_1^2}{2}\Big)\Big(\mathcal{R}_0^s-1\Big)- \frac{2A\beta \epsilon}{\mu_1-\ell}.
\end{align*}
We can choose $\epsilon \leq \frac{\mu_1(\mu_1-\ell)}{4\beta^2A}\Big(\mu_2+\gamma+\frac{\sigma_1^2}{2}\Big)\Big(\mathcal{R}_0^s-1\Big)$, and then we obtain
\begin{align}
\underset{t\to +\infty}{\lim \inf}\frac{1}{t}\int^t_0\E\Big( \beta S(s)I(s)\mathbf{1}_{\Omega_1}\Big)ds&\geq \frac{\mu_1}{2\beta}\Big(\mu_2+\gamma+\frac{\sigma_1^2}{2}\Big)\Big(\mathcal{R}_0^s-1\Big)>0.
\label{22}
\end{align}
Let $q=a_0>1$ be a positive integer and $1<p=\frac{a_0}{a_0-1}$ such that $\chi_2> 0$ and $\frac{1}{q}+\frac{1}{p}=1$. By utilizing the Young inequality $xy\leq \frac{x^p}{p}+\frac{y^q}{q}$ for all $x$,$y>0$, we get
\begin{align*}
\underset{t\to +\infty}{\lim \inf}\frac{1}{t}\int^t_0\E\Big( \beta S(s)I(s)\mathbf{1}_{\Omega_1}\Big)ds&\leq \underset{t\to +\infty}{\lim \inf}\frac{1}{t}\int^t_0\E\bigg(p^{-1}(\eta \beta S(s)I(s))^p+q^{-1}\eta^{-q}\mathbf{1}_{\Omega_1}\bigg)ds\\&\leq \underset{t\to +\infty}{\lim \inf}\frac{1}{t}\int^t_0 \E\Big(q^{-1}\eta^{-q}\mathbf{1}_{\Omega_1}\Big)ds+p^{-1}(\eta \beta)^p\underset{t\to +\infty}{\lim \sup}\frac{1}{t}\int^t_0\E\Big(N^{2p}(s)\Big)ds,
\end{align*}
where $\eta$ is a positive constant satisfying
\begin{align*}
\eta^p\leq\frac{p\mu_1 \chi_2\beta^{-(p+1)}}{8 \chi_1}\Big(\mu_2+\gamma+\frac{\sigma_1^2}{2}\Big)\Big(\mathcal{R}_0^s-1\Big).
\end{align*}
By lemma \ref{L1} and (\ref{22}), we deduce that
\begin{align}
\underset{t\to +\infty}{\lim \inf}\frac{1}{t}\int^t_0\E\big(\mathbf{1}_{\Omega_1}\big)ds&\geq q\eta^q\Bigg(\frac{\mu_1}{2\beta}\Big(\mu_2+\gamma+\frac{\sigma_1^2}{2}\Big)\Big(\mathcal{R}_0^s-1\Big)-\frac{2\chi_1\eta^p\beta^p}{p\chi_2}\Bigg)\nonumber\\&\geq \frac{\mu_1 q\eta^q}{4\beta}\Big(\mu_2+\gamma+\frac{\sigma_1^2}{2}\Big)\Big(\mathcal{R}_0^s-1\Big)>0.
\label{23}
\end{align}
Setting $\Omega_4=\{(S,I,R)\in\R^3_+|\hspace{0.1cm}S\geq \zeta,\hspace{0.1cm}\mbox{or},\hspace{0.1cm} I\geq \zeta\}$ and $\Sigma=\{(S,I,R)\in\R^3_+|\hspace{0.1cm}\epsilon \leq S\leq \zeta,\hspace{0.1cm}\mbox{and},\hspace{0.1cm} \epsilon \leq I\leq \zeta\}$ where $\zeta>0$ is a positive constant to be explained in the following. By using the Tchebychev inequality, we can observe that
\begin{align*}
\E[\mathbf{1}_{\Omega_4}]\leq \P(S(t)\geq \zeta )+\P(I(t)\geq \zeta )&\leq\frac{1}{\zeta}\E[S(t)+I(t)]\leq\frac{1}{\zeta}\bigg(\frac{2A}{\mu_1-\ell}+N(0)\bigg).
\end{align*}
Choosing $\frac{1}{\zeta}\leq \frac{\mu_1 q\eta^q}{8\beta}\Big(\mu_2+\gamma+\frac{\sigma_1^2}{2}\Big)\Big(\mathcal{R}_0^s-1\Big)\Big(\frac{2A}{\mu_1-\ell}+N(0)\Big)^{-1}$. We thus obtain
\begin{align*} 
\underset{t\to +\infty}{\lim \sup}\frac{1}{t}\int^t_0\E[\mathbf{1}_{\Omega_4}]ds&\leq \frac{\mu_1 q\eta^q}{8\beta}\Big(\mu_2+\gamma+\frac{\sigma_1^2}{2}\Big)\Big(\mathcal{R}_0^s-1\Big).
\end{align*}
According to (\ref{23}), one can derive that
\begin{align*}
\underset{t\to +\infty}{\lim \inf}\frac{1}{t}\int^t_0\E[\mathbf{1}_{\Sigma}]ds&\geq \underset{t\to +\infty}{\lim \inf}\frac{1}{t}\int^t_0\E[\mathbf{1}_{\Omega_1}]ds-\underset{t\to +\infty}{\lim \sup}\frac{1}{t}\int^t_0\E[\mathbf{1}_{\Omega_4}]ds\\&\;\;\; \geq\frac{\mu_1 q\eta^q}{8\beta}\Big(\mu_2+\gamma+\frac{\sigma_1^2}{2}\Big)\Big(\mathcal{R}_0^s-1\Big)>0.
\end{align*}
Based on the above analysis, 
we have determined a compact domain $\Sigma\subset \R^3_+$ such that
\begin{align*}
\underset{t\to +\infty}{\lim \inf}\frac{1}{t}\int^t_0\P\Big(s,(S(0),I(0),R(0)),\Sigma\Big)ds\geq \frac{\mu_1 q\eta^q}{8\beta}\Big(\mu_2+\gamma+\frac{\sigma_1^2}{2}\Big)\Big(\mathcal{R}_0^s-1\Big)>0.
\end{align*}
Applying similar arguments to those in \citep{12}, we show the uniqueness of the ergodic stationary distribution of our model (\ref{s2}). This completes the proof.

\end{proof}
%\section{Sufficient conditions for the extinction of the disease}
Now, we will give the result on the extinction of the disease. Define
\begin{align*}
\mathcal{\hat{R}}^s_0=\Big(\mu_2+\gamma+\frac{\sigma_1^2}{2}\Big)^{-1}\bigg(\frac{\beta A}{\mu_1}-\frac{\sigma_2^2A^2}{2\mu_1^2}-\int_Z\eta(u)-\ln(1+\eta(u))\nu(du)\bigg).
\end{align*}
%\begin{align*}
%\hat{R}^s_0=\frac{\beta A}{\mu_1\Big(\mu_2+\gamma+\frac{\sigma_2}{2}\Big)}.
%\end{align*}

\begin{thm}
Let $(S(t),I(t),R(t))$ be the solution of system (\ref{s2}) with initial value $(S(0), I(0), R(0))\in\R^3_+$. If 
\begin{align}
&\mathcal{\hat{R}}_0^s<1\hspace{0.2cm}\mbox{and} \hspace{0.2cm} \sigma_2^2\leq \frac{\mu_1\beta}{A},&
\label{cond1}
\end{align}
or
\begin{align}
&\frac{\beta^2}{2\sigma_2^2}-\Big(\mu_2+\gamma+\frac{\sigma^2_1}{2}\Big)-\int_Z\eta(u)-\ln(1+\eta(u))\nu(du)<0.&
\label{cond2}
\end{align}
% Suppose that one of
%the following two conditions holds:
Then, the disease dies out exponentially with probability one. That is to say,
\begin{align}
\underset{t\to \infty}{\lim \sup}\frac{\ln I(t)}{t}<0\hspace*{0.3cm}\mbox{a.s.}
\label{ext}
\end{align}
\label{thm2}
\end{thm}

\begin{proof}
By Itô’s formula for all $t \geq0$, we have
\begin{align}
d\ln I(t)&=\Big(\beta S(t)-\frac{\sigma_2^2}{2}S^2(t)-\Big(\mu_2+\gamma+\frac{\sigma^2_1}{2}\Big)-\int_Z\eta(u)-\ln(1+\eta(u))\nu(du)\Big)dt\nonumber\\&\;\;\;+\sigma_1 dW_1(t)+\sigma_2S(t)dW_2(t)+\int_Z \ln(1+\eta(u))\widetilde{N}(dt,du).
\label{25}
\end{align}
Integrating (\ref{25}) from $0$ to $t$ and then dividing $t$ on both sides, we get
\begin{align}
\frac{\ln I(t)}{t}&= \frac{\beta}{t}\int^t_0S(s)ds-\frac{\sigma_2^2}{2t}\int^t_0S^2(s)ds-\Big(\mu_2+\gamma+\frac{\sigma^2_1}{2}\Big)\\&\;\;\;-\int_Z\eta(u)-\ln(1+\eta(u))\nu(du)+\Phi_2(t),
\label{28}
\end{align}
where
\begin{align*}
\Phi_2(t)=\frac{\ln I(0)}{t}+\frac{\sigma_1}{t} W_1(t)+\frac{\sigma_2}{t}\int^t_0 S(s)dW_2(s)+\frac{1}{t}\int_0^t\int_Z \ln(1+\eta(u))\widetilde{N}(ds,du).
\end{align*}
Obviously, we know that
\begin{align*}
\frac{1}{t}\int^t_0S^2(s)ds\geq \Big(\frac{1}{t}\int^t_0S(s)ds\Big)^2.
\end{align*}
Therefore we derive
\begin{align*}
\frac{\ln I(t)}{t}&\leq \frac{\beta}{t}\int^t_0S(s)ds-\frac{\sigma_2^2}{2}\Big(\frac{1}{t}\int^t_0S(s)ds\Big)^2-\Big(\mu_2+\gamma+\frac{\sigma^2_1}{2}\Big)-\int_Z\eta(u)-\ln(1+\eta(u))\nu(du)+\Phi_2(t)\\&\leq \beta \bigg(\frac{A}{\mu_1}-\frac{(\mu_2+\gamma)}{\mu_1 t}\int^t_0I(s)ds+\Phi_1(t)\bigg)-\frac{\sigma_1^2}{2}\bigg(\frac{A}{\mu_1}-\frac{(\mu_2+\gamma)}{\mu_1 t}\int^t_0I(s)ds+\Phi_1(t)\bigg)^2\\&\;\;\; -\Big(\mu_2+\gamma+\frac{\sigma^2_1}{2}\Big)-\int_Z\eta(u)-\ln(1+\eta(u))\nu(du)+\Phi_2(t).
\end{align*}
Hence one can see that
\begin{align}
\frac{\ln I(t)}{t}&\leq\frac{\beta A}{\mu_1}-\frac{A^2\sigma_1^2}{2\mu_1^2}-\Big(\mu_2+\gamma+\frac{\sigma^2_1}{2}\Big)-\int_Z\eta(u)-\ln(1+\eta(u))\nu(du)\nonumber\\&\;\;\;-\frac{(\mu_2+\gamma)}{\mu_1}\bigg(\beta-\frac{A\sigma_1^2}{\mu_1}\bigg)\frac{1}{t}\int^t_0I(s)ds
-\frac{\sigma_1^2}{2}\bigg(\frac{(\mu_2+\gamma)}{\mu_1 t}\int^t_0I(u)du\bigg)^2\nonumber\\&\;\;\;+\Phi_2(t)+\Phi_3(t),
\label{32}
\end{align}
where
\begin{align*}
\Phi_3(t)=\beta\Phi_1(t)-\frac{\sigma^2_1}{2}\Phi^2_1(t)-\frac{\sigma_1^2A\Phi_1(t)}{\mu_1}+\sigma_1^2\Phi_1(t)\frac{(\mu_2+\gamma)}{\mu_1 t}\int^t_0I(s)ds.
\end{align*}
An application of large number theorem for martingales, one has
\begin{align*}
\underset{t\to \infty}{\lim }\frac{\Phi_2(t)}{t}=\underset{t\to \infty}{\lim }\frac{\Phi_3(t)}{t}=0\hspace*{0.3cm}\mbox{a.s.}
\end{align*}
Taking the superior limit on both sides of (\ref{32}), then by condition (\ref{cond1}), we arrive at
\begin{align*}
\underset{t\to \infty}{\lim \sup}\frac{\ln I(t)}{t}&\leq \Big(\mu_2+\gamma+\frac{\sigma_2}{2}\Big)\Big(\mathcal{\hat{R}}^s_0-1\Big)<0\hspace*{0.3cm}\mbox{a.s.}
\end{align*}
If the condition (\ref{cond2}) is satisfied, then
\begin{align*}
\frac{\ln I(t)}{t}&=\frac{\beta}{t}\int^t_0S(s)ds-\frac{\sigma_2^2}{2t}\int^t_0S^2(s)ds-\Big(\mu_2+\gamma+\frac{\sigma^2_1}{2}\Big)-\int_Z\eta(u)-\ln(1+\eta(u))\nu(du)\\&\;\;\;+\frac{\ln I(0)}{t}+\frac{\sigma_1}{t} W_1(t)+\frac{\sigma_2}{t}\int^t_0 S(s)dW_2(s)+\frac{1}{t}\int_0^t\int_Z \ln(1+\eta(u))\widetilde{N}(ds,du)\\&=\frac{\beta^2}{2\sigma_2^2}-\Big(\mu_2+\gamma+\frac{\sigma^2_1}{2}\Big)-\int_Z\eta(u)-\ln(1+\eta(u))\nu(du)-\frac{1}{t}\int^t_0\bigg(\frac{\sigma_2^2}{2}\bigg(S(s)-\frac{\beta}{\sigma_2^2}\bigg)^2\bigg)ds\\&\;\;\;+\frac{\ln I(0)}{t}+\frac{\sigma_1}{t} W_1(t)+\frac{\sigma_2}{t}\int^t_0 S(s)dW_2(s)+\frac{1}{t}\int_0^t\int_Z \ln(1+\eta(u))\widetilde{N}(ds,du)\\&\leq 
\frac{\beta^2}{2\sigma_2^2}-\Big(\mu_2+\gamma+\frac{\sigma^2_1}{2}\Big)-\int_Z\eta(u)-\ln(1+\eta(u))\nu(du)+\frac{\ln I(0)}{t}+\frac{\sigma_1}{t} W_1(t)\\&\;\;\;+\frac{\sigma_2}{t}\int^t_0 S(s)dW_2(s)+\frac{1}{t}\int_0^t\int_Z \ln(1+\eta(u))\widetilde{N}(ds,du).
\end{align*}
%\begin{align*}
%\underset{t\to \infty}{\lim \sup}\frac{\ln I(t)}{t}&\leq \frac{\beta^2}{2\sigma_1^2}-\Big(\mu_2+\gamma+\frac{\sigma^2_2}{2}\Big)<0\hspace*{0.3cm}\mbox{a.s.}
%\end{align*}
By the large number theorem for martingales and the condition (\ref{cond2}), our desired result (\ref{ext}) holds true. This completes the proof.

\end{proof}
%\section{Simulations and conclusion}
\section{Examples}
In this section, we will validate our theoretical results with the help of numerical simulation examples taking parameters from the theoretical data mentioned in the table \ref{value}. We numerically simulate the solution to system (\ref{s2}) with initial value $(S(0), I(0), R(0)) = (0.4, 0.3, 0.1)$. For the purpose of showing the effects of the perturbations on the disease dynamics, we have realized the simulation $15000$ times. 
\begin{center}
\begin{tabular}{llll}
\hline
Parameters  \hspace*{0.5cm}&Description  \hspace*{0.5cm}&Value  \\
\hline
$A$ & The recruitment rate & 0.09 \\
$\mu_1$ & The natural mortality rate & 0.05\\
$\beta$ & The transmission rate &0.06 \\
$\gamma$ & The recovered rate &0.01 \\
$\mu_2$ &The general mortality &0.09 \\
\hline
\end{tabular}
\captionof{table}{Some theoretical parameter values of the model (\ref{s2}).}
\label{value}
\end{center}
\begin{figure}[!htb]
\subcaptionbox{The left figure is the stationary distribution for S(t), the right picture is the stationary distribution I(t).}
{\includegraphics[width=3.2in]{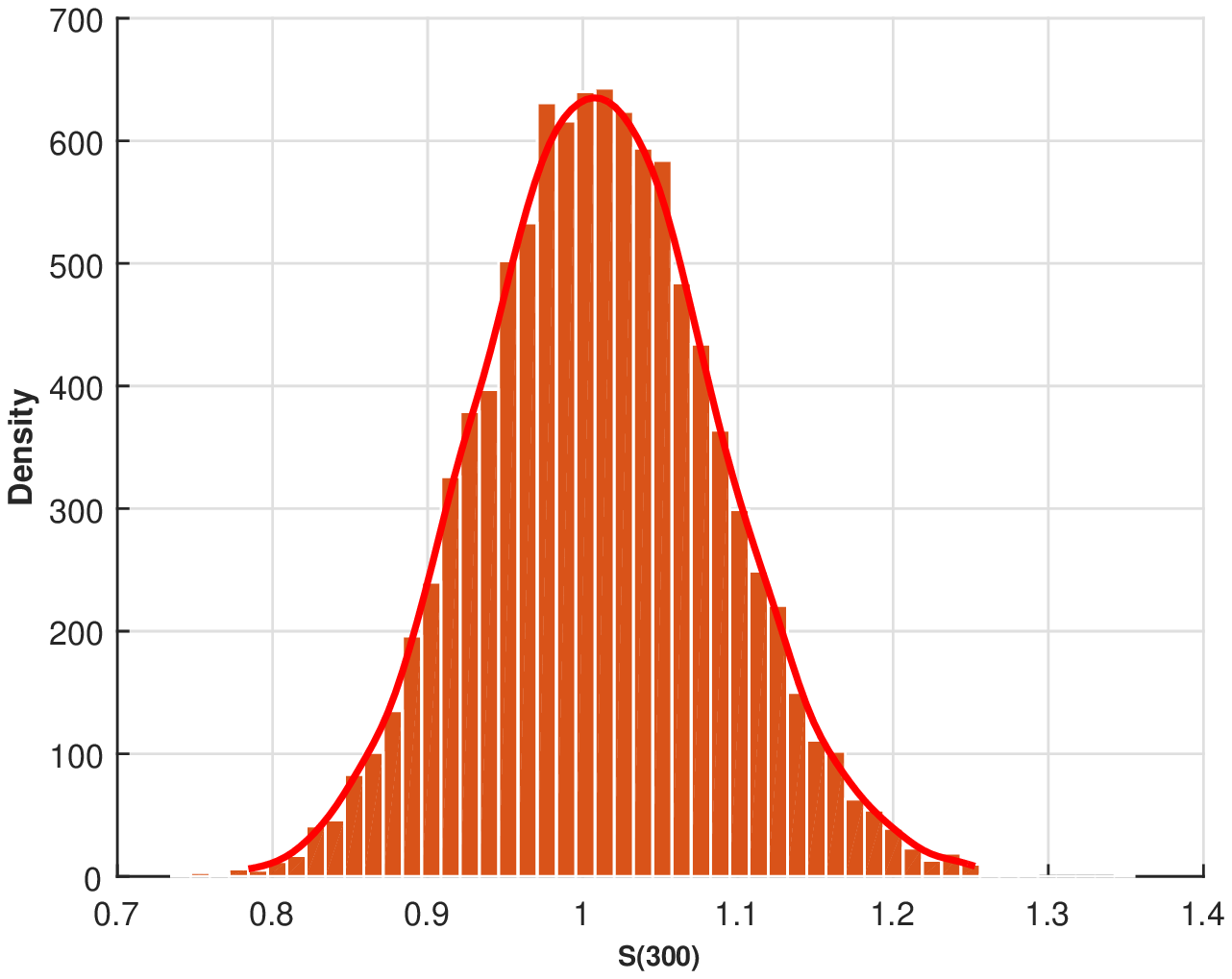} 
\includegraphics[width=3.2in]{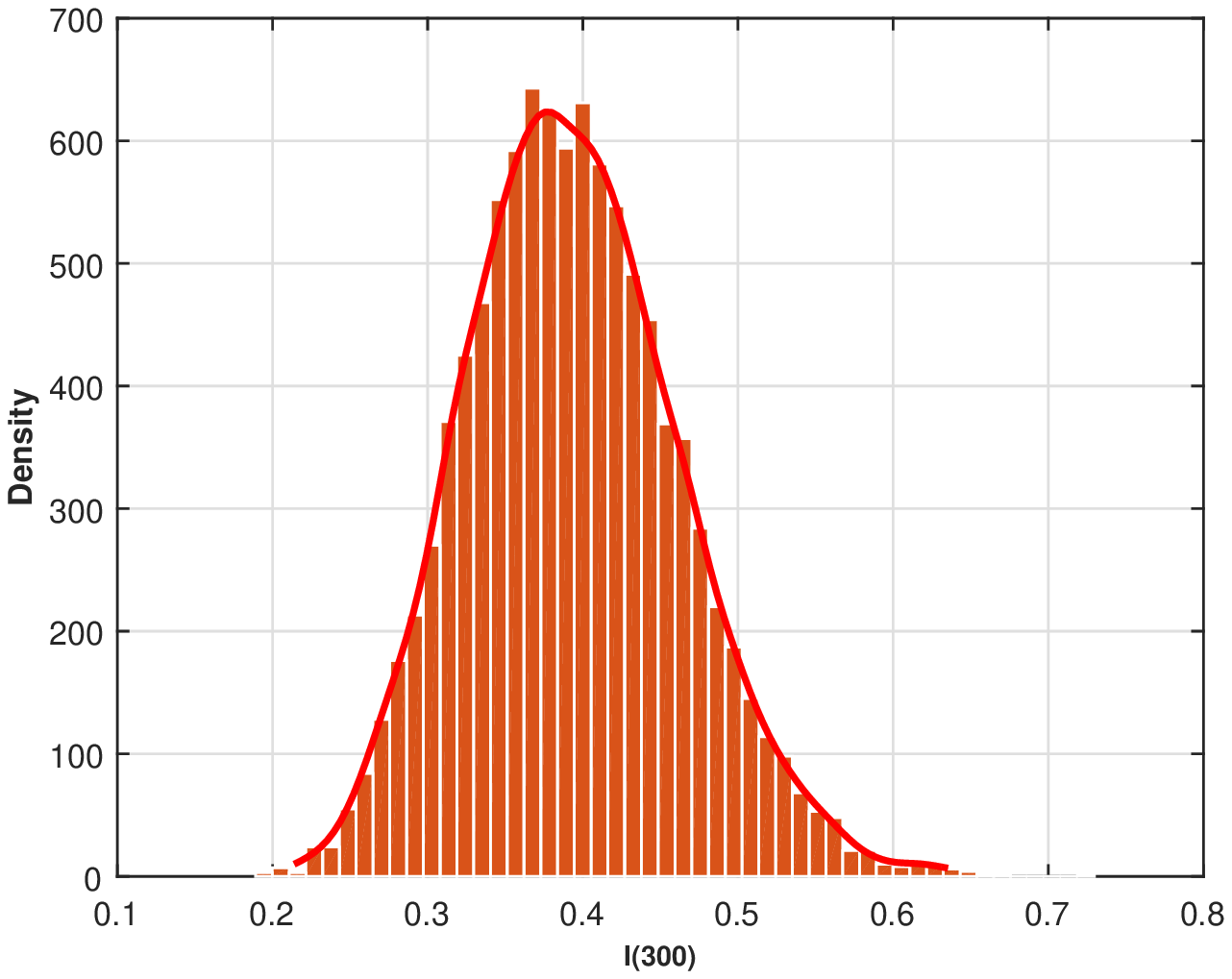}} 
\subcaptionbox{The left figure is the stationary distribution for R(t), the right picture is the trajectory of the solution.}
{\includegraphics[width=3.2in]{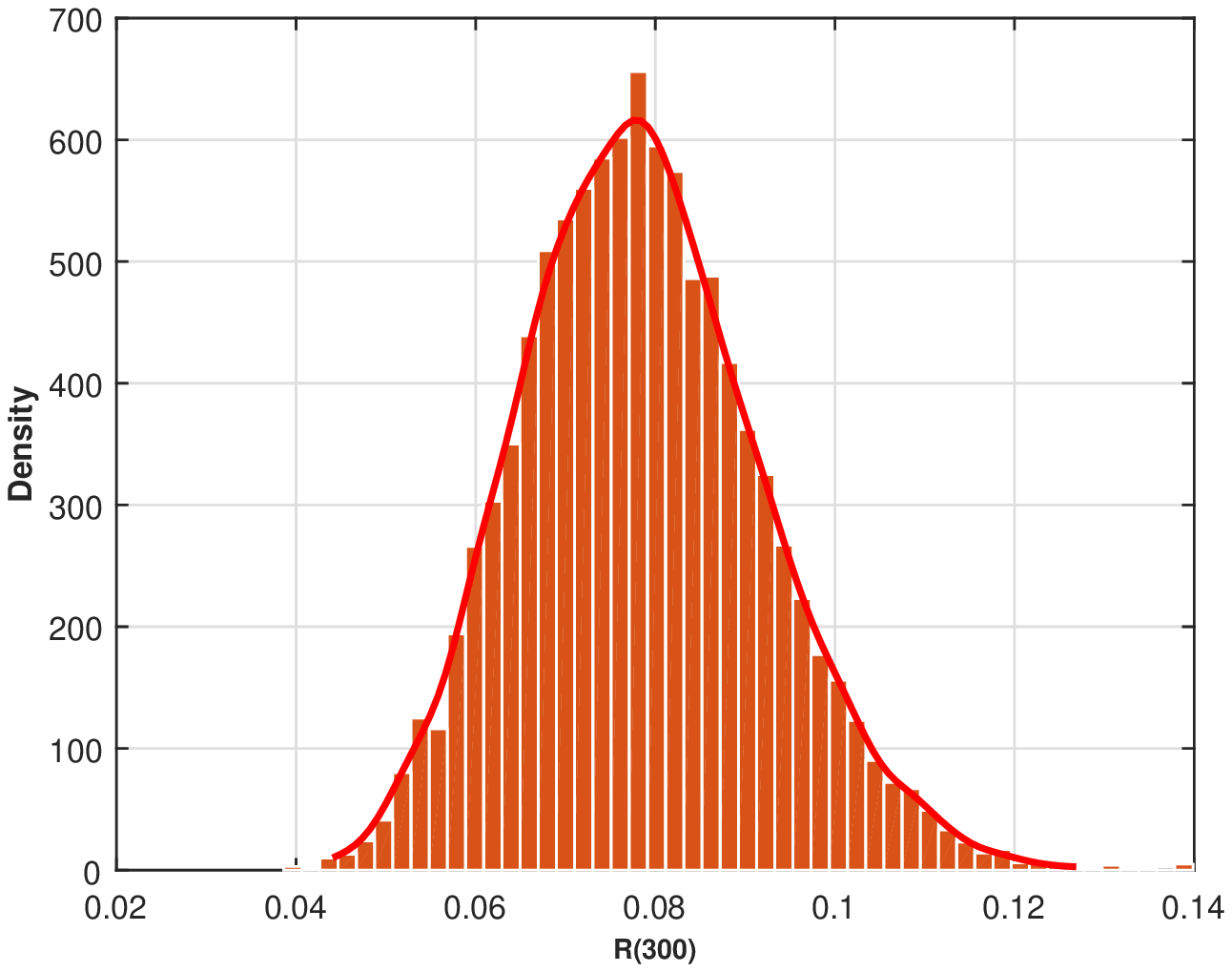} 
\includegraphics[width=3.2in]{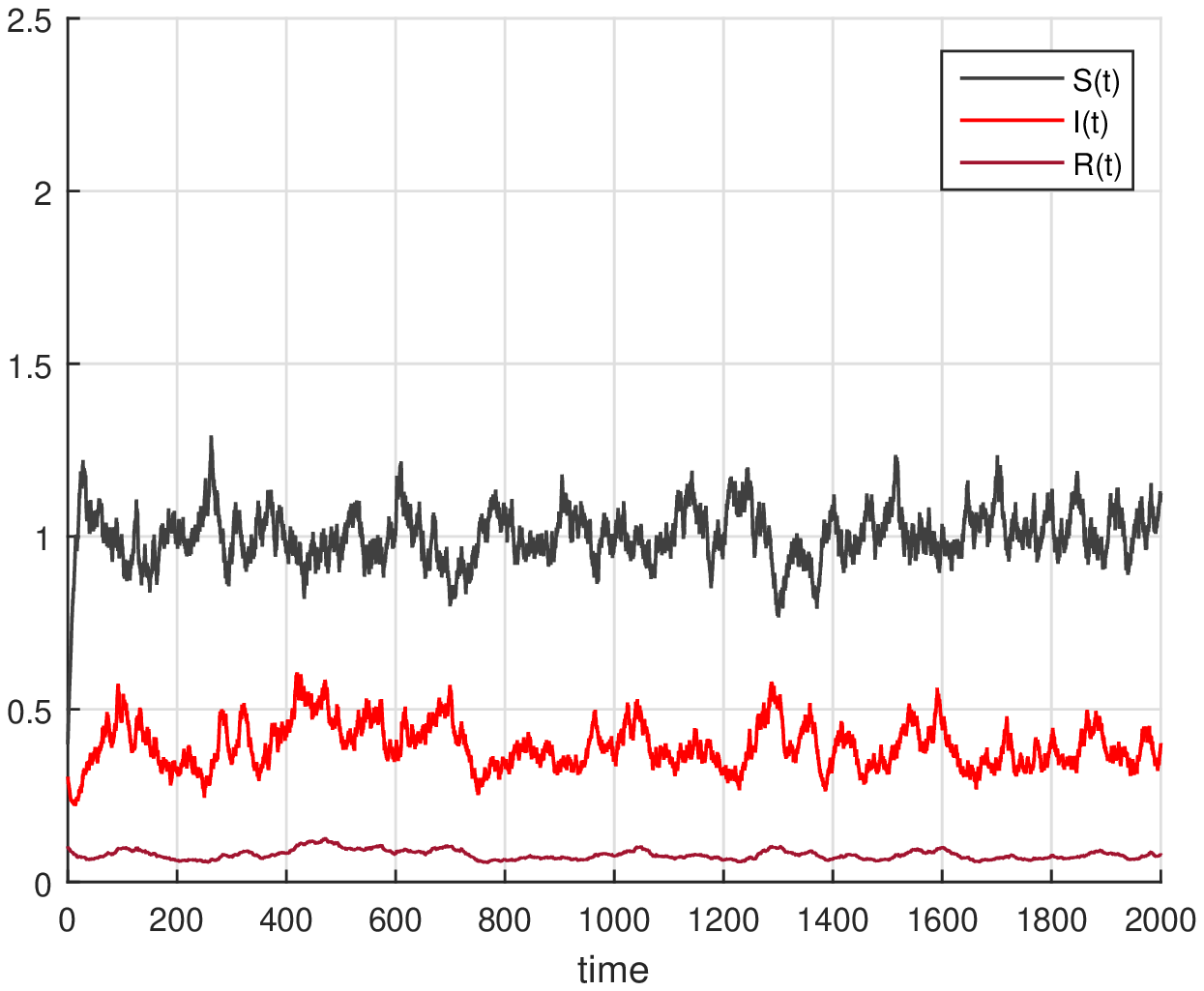} } 
\caption{The numerical illustration of obtained results in the theorem \ref{thm1}.}
\label{fig1}
\end{figure}
\begin{ex}
We have chosen the stochastic fluctuations intensities $\sigma_1=0.03$ and $\sigma_2=0.02$. Furthermore, we assume that $\eta(u)=0.05$, $Z=(0,\infty)$ and $\nu(Z)=1$. Then, $R_0^s=1.672>1$. From figure \ref{fig1}, we show the existence of the unique stationary distributions for $S(t)$, $I(t)$ and $R(t)$ of model (\ref{s2}) at $t = 300$, where the smooth curves are the probability density functions of $S(t)$, $I(t)$ and $R(t)$, respectively. It can be obviously observed that the solution of the stochastic model (\ref{s2}) persists in the mean. 
\end{ex}
\begin{figure*}[!htb]
\begin{center}$
\begin{array}{cc}
\includegraphics[width=5.5in]{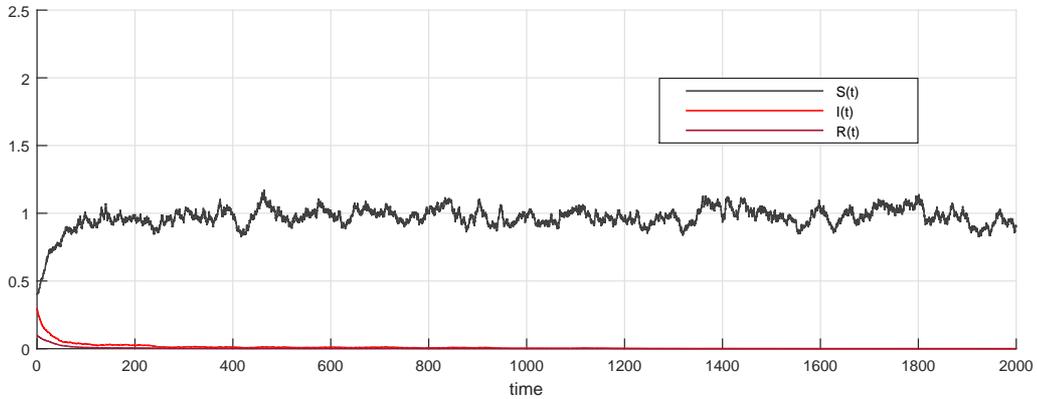}  
\end{array}$
\end{center}
\caption{The numerical simulation of the solution $(S(t),I(t),R(t))$ in system (\ref{s2}).}
\label{fig2}
\end{figure*}
\begin{ex}
Now, we choose the white noise intensities $\sigma_1=0.2$ and $\sigma_2=0.3$ to ensure that the condition (\ref{cond2}) of theorem (\ref{thm2}) is satisfied. We can conclude that for any initial value, $I(t)$ obeys
\begin{align*}
\underset{t\to \infty}{\lim \sup}\frac{1}{t}\ln\frac{I(t)}{I(0)}&\leq \frac{\beta^2}{2 \sigma_2^2}-\Big(\mu_2+\gamma+\frac{\sigma_1^2}{2}\Big)-\int_Z\eta(u)-\ln(1+\eta(u))\nu(du)=-0.065<0 \;\;\; \mbox{a.s.}
\end{align*}
That is, $I(t)$ will tend to zero exponentially with probability one (see figure \ref{fig2}). To verify that the condition (\ref{cond1}) is satisfied, we change $\sigma_1$ to $0.01$, $\sigma_2$ to $0.02$, $\mu_2$ to $0.43$ and $\beta$ to $0.145$ and keep other parameters unchanged. Then we have
\begin{align*}
\mathcal{\hat{R}}^s_0=\Big(\mu_2+\gamma+\frac{\sigma_1^2}{2}\Big)^{-1}\bigg(\frac{\beta A}{\mu_1}-\frac{\sigma_2^2A^2}{2\mu_1^2}-\int_Z\eta(u)-\ln(1+\eta(u))\nu(du)\bigg)= 0.9860<1,
\end{align*}
and
\begin{align*}
\sigma_1^2-\frac{\mu_1\beta}{A}=-0.0804<0.
\end{align*}
Therefore, the condition (\ref{cond1}) of theorem \ref{thm1} is satisfied. We can conclude that for any initial value, $I(t)$ obeys
\begin{align*}
\underset{t\to \infty}{\lim \sup}\frac{1}{t}\ln\frac{I(t)}{I(0)}&\leq (\mathcal{\hat{R}}^s_0-1)\Big(\mu_2+\gamma+\frac{\sigma_2^2}{2}\Big)=-0.0061<0 \;\;\; \mbox{a.s.}
\end{align*}
That is, $I(t)$ will tend to zero exponentially with probability one (see figure \ref{fig3}).
\end{ex}
\begin{figure*}[!htb]
$
\begin{array}{cc}
\includegraphics[width=5.5in]{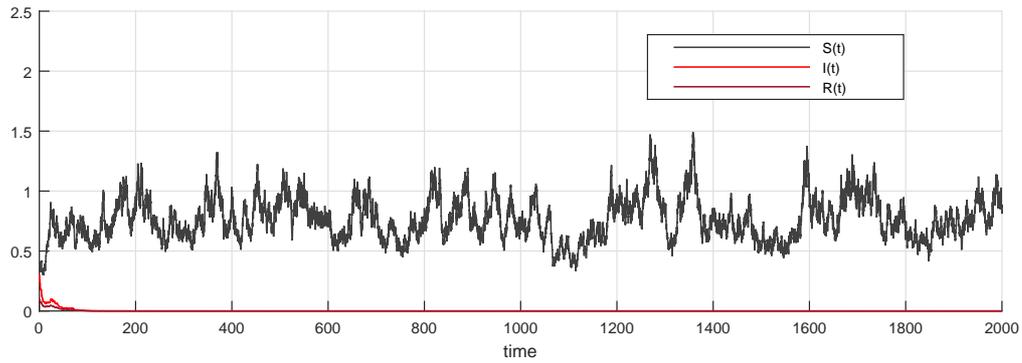}  
\end{array}$
\caption{The numerical simulation of the solution $(S(t),I(t),R(t))$ in system (\ref{s2}).}
\label{fig3}
\end{figure*}
\section{Conclusion}
The dissemination of the epidemic diseases presents
a global issue that concerns decision-makers to elude deaths and deterioration of economies.  Many scientists are motivated to understand and suggest the ways for diminishing the epidemic dissemination. The first generation proposed the deterministic models that showed a lack of realism due to the neglecting of environmental perturbations.  Recent studies present a deep understanding of the process of outbreak diseases by taking into account their random aspect. This contribution is the first work that combines two different disturbances: white and Lévy noises. This original idea generalizes the existing works. Our work based on the following new techniques:
\begin{enumerate}
\item The calculation of the temporary average of the solution of the auxiliary equation (\ref{s4}) instead of the classic method based on the explicit form of the stationary distribution in the model (\ref{s4}).
\item The investigation of the disease persistence with a new approach based on the stochastic comparison theorem.
\item The use of Feller property and mutually exclusive possibilities lemma for proving the ergodicity of the model (\ref{s2}).
\end{enumerate}
Based on the above techniques, our analysis leads to three main results:
\begin{enumerate}
\item In theorem \ref{thm1}, we proved that the persistence in the mean of the disease occurs under the same condition of the existence of a unique ergodic stationary distribution.
\item In theorem \ref{thm2}, we showed that the extinction of the disease in the stochastic system (\ref{s2}) occurs if
one of the conditions (\ref{cond1}) and (\ref{cond2}) holds. It should be noted that these conditions are sufficient for
the extinction of the epidemic.
\end{enumerate}
Comparing our stochastic model with corresponding previous researches, our theoretical analysis leads to establishing a new appropriate condition for the persistence and the existence of ergodic stationary distribution in the model (\ref{s2}). However, our paper brings more challenges to propose an improved method to obtain the global threshold between the existence of the unique ergodic  stationary distribution (persistence) and the extinction of
a disease. We seek in our future works to treat this interesting problem.

%\begin{figure*}[!]
%\begin{center}$
%\begin{array}{cc}
%\includegraphics[width=5.5in]{extin1.eps} \\
%\includegraphics[width=5.5in]{extin2.eps}   
%\end{array}$
%\end{center}
%\caption{Compute simulation of the paths of the solution for the SIRS epidemic model \ref{s3} (with jump), the trajectories of the system \ref{s2} (without jump) and the solution of the corresponding deterministic system \ref{s1}.}
%\label{fig3}
%\end{figure*}
\section*{References}
%\nocite{*}
\bibliographystyle{ieeetr}
\bibliography{double}
\end{document}